\documentclass[12pt]{amsart}
\usepackage[utf8]{inputenc}

\usepackage{mathrsfs} 
\usepackage{amssymb}
\usepackage{bm}
\usepackage{graphicx}
\usepackage[centertags]{amsmath}
\usepackage{amsfonts}
\usepackage{amsthm}
\usepackage{enumerate}
\usepackage{tocvsec2}
\usepackage{xcolor}
\usepackage[margin=.5in]{geometry}

\newtheorem{theorem}{Theorem}[section]
\newtheorem*{theorem*}{Theorem}

\newtheorem{corollary}[theorem]{Corollary}
\newtheorem{lemma}[theorem]{Lemma}
\newtheorem{rem}[theorem]{Remark}

\newtheorem{proposition}[theorem]{Proposition}

\newtheorem{fact}[theorem]{Fact}
\newtheorem*{fact*}{Fact}

\theoremstyle{definition}
\newtheorem{definition}[theorem]{Definition}

\newcommand{\ee}{\varepsilon}
\newcommand{\nn}{\mathbb{N}}

\begin{document}

\title[A short calculation]{A short calculation of the Szlenk index of spaces of continuous functions}
\author{R.M. Causey}

\begin{abstract} We provide a short calculation of the Szlenk index of $C(K)$ spaces using the Grasberg norm.
\end{abstract}

\maketitle

\section{The ordinal connection}

Throughout, $\textbf{Ord}$ denotes the class of ordinal numbers. We recall the convention that for $\xi\in \textbf{Ord}$, $\xi=\{\zeta\in \textbf{Ord}:\zeta<\xi\}$. We also recall the convention that $\xi<\textbf{Ord}$ for all $\xi\in \textbf{Ord}$, so that $<$ and $\in$ can be used interchangeably. We let $\xi+=\xi+1=[0,\xi]$.   

The Szlenk index is a fundamental object in Banach space theory. One example of the utility of the Szlenk index is its use in one direction of the Bessaga Pe\l cz\'{n}ski classification of $C(K)$ spaces, $K$ countable, compact, Hausdorff.  Samuel \cite{S} calculated the Szlenk index of $C(\xi+)$ for countable $\xi$.  In particular, he proved that the Szlenk index of $C(\zeta+)=\omega^{\xi+1}$ for countable $\zeta\in [\omega^{\omega^\xi},\omega^{\omega^{\xi+1}})$. Samuel's proof used the opposite direction of the Bessaga-Pe\l cz\'{n}ski result. Brooker \cite{B} extended this result to arbitrary ordinals.  In the absence of the Bessaga-Pe\l cz\'{n}ski result for uncountable ordinals, for which the analogous statement is known to be false, Brooker's result depended on representations of such spaces as $c_0$ direct sums in terms of lower ordinal spaces, combined with his previous,  technical results concerning the Szlenk indices of infinite $c_0$ direct sums.  Both of these results depended on the ordinal interval structure of the space $K$. Since the Szlenk index is a quantification of the Asplundness of a Banach space, $C(K)$ is Asplund if and only if $K$ is scattered, and the scatteredness of $K$ is measured by the Cantor-Bendixson index $CB(K)$ of $K$, the natural relationship should be between the Cantor-Bendixson index of $K$ and the Szlenk index of $C(K)$. The goal of this work is to give a non-technical, short calculation of the Szlenk index of arbitrary $C(K)$ spaces which directly uses the relationship between the Cantor-Bendixson index of $K$ by using the Grasberg norm. The Grasberg norm has already been used in a similar way to prove a number of results on  $C(K)$ spaces, namely their subprojectivity \cite{C3} and the Szlenk indices of their projective tensor products \cite{CGS1},\cite{CGS2}. 

\begin{theorem} Let $K$ be a compact, Hausdorff topolgical space.   Then $C(K)$ is Asplund if and only if $K$ is scattered. In this case, if $\xi$ is the minimum ordinal such that $CB(K)\leqslant \omega^\xi$, then $Sz(C(K))=\omega^\xi$. 
\label{main}
\end{theorem}

We observe the following consequence, which recovers Brooker's result. 

\begin{corollary} Let $\zeta$ be a infinite ordinal and let $\xi$ be such that $\omega^{\omega^\xi}\leqslant \zeta < \omega^{\omega^{\xi+1}}$.  Then $Sz(C(\zeta+))=\omega^{\xi+1}$. 

\end{corollary}

\begin{proof} By Theorem \ref{main}, it suffices to show that $\omega^\xi<CB(\zeta+)\leqslant \omega^{\xi+1}$.

We first observe that for any ordinals $\beta\leqslant \alpha$, 

\[[1,\omega^\alpha]^\beta = \{\omega^\alpha\}\cup \{\omega^{\ee_1}+\ldots + \omega^{\ee_n}:n\in\nn, \alpha > \ee_1\geqslant \ldots \geqslant \ee_n\geqslant\beta\},\] which can be easily shown by induction. Therefore $CB(\omega^{\omega^\alpha}+)=\omega^\alpha+1$.     From this it follows that $\omega^{\omega^\xi}\in \zeta+^{\omega^\xi}$, and $CB(\zeta+)>\omega^\xi$.   On the other hand, there exists $m\in\nn$ such that $\zeta<\omega^{\omega^\xi m}$, so  \[CB(\zeta+)\leqslant CB(\omega^{\omega^\xi m}+)=\omega^\xi m<\omega^{\xi+1}.\]
    
\end{proof}

\subsection{The Szlenk index}

For a Banach space $X$, $\ee>0$, and a weak$^*$-compact subset $K$ of $X^*$, we define $s_\ee(K)$ to be the set of all $x^*\in K$ such that for all weak$^*$-neighborhoods $V$ of $x^*$, $\text{diam}(K\cap V)>\ee$.  We define \[s^0_\ee(K)=K,\] \[s^{\xi+1}_\ee(K)=s_\ee(s^\xi_\ee(K)),\] and if $\xi$ is a limit ordinal, \[s^\xi_\ee(K)=\bigcap_{\zeta<\xi}s_\ee^\zeta(K).\] Note that each $s_\ee^\xi(K)$ is again weak$^*$-compact.   

We define \[Sz(K,\ee)=\{\xi\in \textbf{Ord}:s^\xi_\ee(K)\neq \varnothing\}\] and \[Sz(K)=\bigcup_{\ee>0}Sz(K,\ee).\] This is equivalent to defining $Sz(K,\ee)=\textbf{Ord}$ if $s^\xi_\ee(K)\neq \varnothing$ for all $\xi\in \textbf{Ord}$, and otherwise $Sz(K,\ee)=\min \{\xi\in \textbf{Ord}:s^\xi_\ee(K)=\varnothing\}$.   Similarly, if $Sz(K,\ee)\in \textbf{Ord}$ for all $\ee>0$, then $Sz(K)=\sup_{\ee>0}Sz(K,\ee)$.   

It is quite clear that $Sz(K)\in \textbf{Ord}$ if and only if for every $\ee>0$ and every non-empty subset $L$ of $K$, there exists a weak$^*$-open set $V$ such that $L\cap V\neq \varnothing$ and $\text{diam}(L\cap V)\leqslant\ee$.  In this case, $L$ is said to be \emph{weak}$^*$-\emph{fragmentable}.   We say $X$ is \emph{Asplund} if $B_{X^*}$ is weak$^*$-fragmentable, which is equivalent to the condition that $Sz(X):=Sz(B_{X^*})\in \textbf{Ord}$.  This definition is not the original definition of an Asplund space, but it is known to be equivalent \cite{Ba}. 

We conclude this section with a result about the Szlenk index from \cite{L}.

\begin{proposition} For an Asplund Banach space $X$, there exists an ordinal $\xi$ such that $Sz(X)=\omega^\xi$. 
\label{form}
\end{proposition}

\subsection{Trees, rank, and weak nullity}

For us, a \emph{tree} is a partially ordered set $(T, \preceq)$ such that for each $t\in T$, $\{s\in T:s\preceq t\}$ is finite and linearly ordered.  We will also assume that our trees never contain the node $\varnothing$, so that we can append the node to obtain $\{\varnothing\}\cup T$, and we extend the order $\preceq$ to $\{\varnothing\}\cup T$ by declaring that $\varnothing\prec t$ for all $t\in T$.   Thus we have appended a root to $T$.    

We let $MAX(T)$ denote the $\preceq$-maximal members of $T$.   We define $T'=T\setminus MAX(T)$. Note that $T'$ is also a tree. We also define \[T^0=T,\] \[T^{\xi+1}=(T^\xi)',\] and if $\xi$ is a limit ordinal, \[T^\xi=\bigcap_{\zeta<\xi}T^\zeta.\] We  define \[\text{rank}(T)=\{\xi\in \textbf{Ord}:T^\xi\neq \varnothing\}.\] If there exists $\xi$ such that $T^\xi=\varnothing$ (equivalently, if $\text{rank}(T)\in \textbf{Ord}$), then we say $T$ is \emph{well-founded}, and otherwise we say $T$ is \emph{ill-founded}.   In that case, $\text{rank}(T)=\textbf{Ord}$.   If $T$ is well-founded, then $\text{rank}(T)=\min\{\xi\in \textbf{Ord}:T^\xi=\varnothing\}$. 

For a tree $T$, a Banach space $X$, and a collection $(x_t)_{t\in T}\subset X$, we say the collection $(x_t)_{t\in T}$ is \emph{weakly null} if for each ordinal $\xi$ and each $t\in (\{\varnothing\}\cup T)^{\xi+1}$, \[0\in \overline{\{x_s:s\in T^\xi,s^-=t\}}^\text{weak}.\] Here, $s^-=t$ denotes the relation that $t$ is the immediate predecessor of $s$. That is, $t$ is the (necessarily unique) maximum of $\{r\in T^\xi:r\prec s\}$.   

We next isolate some easy facts about trees.

\begin{fact}  Let $T$ be a well-founded tree and let $\xi$ be an ordinal such that $\text{\emph{rank}}(T)\geqslant \xi$.  Let $X$ be a Banach space and let $(x_t)_{t\in T}\subset X$ be weakly null. 
\begin{enumerate}[(i)]\item $\text{\emph{rank}}(T\setminus T^\xi)=\xi$ and $(x_t)_{t\in T\setminus T^\xi}$ is weakly null. 

\item For $s\in MAX(T^\xi)$, if $S=\{t\in T:s\prec t\}$, $\text{\emph{rank}}(S)=\xi$ and $(x_s)_{t\in S}$ is weakly null. 

\end{enumerate}
\end{fact}

\begin{proof}

$(i)$ An easy proof by induction yields that $(T\setminus T^\xi)^\mu=T^\mu\setminus T^\xi$ for any $\mu\leqslant \xi$. Then for $\mu\leqslant \xi$, $(T\setminus T^\xi)^\mu=T^\mu\setminus T^\xi=\varnothing $ if and only if $\mu=\xi$. This yields that $\text{rank}(T\setminus T^\xi)=\xi$.  

It is quite clear that for any $\zeta$ and $t\in (\{\varnothing\}\cup (T\setminus T^\xi))^{\zeta+1}$, \[\{u\in (T\setminus T^\xi)^\zeta: u^-=t\} = \left\{\begin{array}{ll}  \bigcup\{u\in (T\setminus T^\xi)^\zeta:u^-\in MAX(T^\xi)\} & : t = \varnothing \\  \{u\in (T\setminus T^\xi)^\zeta:u^-=t\} & : t \neq \varnothing. \end{array}\right. \] Therefore \[0\in \overline{\{x_u:u\in (T\setminus T^\xi)^\zeta, u^-=t\}}^\text{weak}.\]

$(ii)$    We first claim that for $\mu\leqslant \xi\leqslant \text{rank}(T)$ and $s\in (\{\varnothing\}\cup T)^\xi$, \[\{t\in T:s\prec t\}^\mu = \{t\in T^\mu: s\prec t\}.\]  We prove this by induction on $\mu$ for $\xi$ held fixed. The case $\mu=0$ is clear, and the case in which $\mu$ is a limit ordinal easily follows by taking intersections. Assume the result holds for $\mu$. Fix  \[u\in \{t\in T:s\prec t\}^{\mu+1}=(\{t\in T:s\prec t\}^\mu)'= \{t\in T^\mu:s\prec t\}'.\]  Then there exists $v\in \{t\in T^\mu:s\prec t\}$ such that $u\prec v$, so $s\prec u\prec v\in T^\mu$, and $u\in \{t\in T^{\mu+1}:s\prec t\}$.  On the other hand, assume $u\in \{t\in T^{\mu+1}:s\prec t\}$.  Since $u\in T^{\mu+1}$, there exists $v\in T^\mu$ such that $u\prec v$. Necessarily $s\prec u\prec v$, so $u\prec v\in \{t\in T^\mu:s\prec t\}$, and \[u\in (\{t\in T^\mu:s\prec t\})'= \{t\in T:s\prec t\}^{\mu+1}.\] Since $s\in MAX(T^\xi)$ and $\mu\leqslant \xi$, $\{t\in T^\mu:s\prec t\}$ is empty if and only if $\mu=\xi$. By the preceding paragraph, $\{t\in T^\mu:s\prec t\}=S^\mu$, so $\text{rank}(S)=\xi$. 

It is quite clear that for any $\zeta$ and $t\in (\{\varnothing\}\cup S)^{\zeta+1}$, \[\{u\in S^\zeta: u^-=t\} = \left\{\begin{array}{ll} \{u\in S^\zeta:u^-=s\} & : t = \varnothing \\ \{u\in S^\zeta:u^-=t\} & : t \neq \varnothing. \end{array}\right. \] Therefore \[0\in \overline{\{x_u:u\in S^\zeta, u^-=t\}}^\text{weak}.\]

\end{proof}

We conclude this section with the following fact concerning the Szlenk index from \cite{C0}. 

\begin{theorem} For a Banach space $X$ and an ordinal $\xi$, $Sz(X)\leqslant\xi $ if and only if for any tree $T$ with $\text{\emph{rank}}(T)=\xi$ and any weakly null collection $(x_t)_{t\in T}\subset B_X$, \[0=\inf\{\|x\|:t\in T, x\in \text{\emph{co}}\{x_s:s\preceq t\}\}.\]
\label{upper}
\end{theorem}

\subsection{Cantor-Bendixson index}

For a topological space $K$ and a subset $L$ of $K$, we let $L'$ denote the set of points in $L$ which are not isolated relative to $L$.  We define \[L^0=L,\] \[L^{\xi+1}=(L^\xi)',\] and if $\xi$ is a limit ordinal \[L^\xi=\bigcap_{\zeta<\xi}L^\zeta.\] We let \[CB(L)=\{\xi\in \textbf{Ord}:L^\xi\neq \varnothing\}.\] We recall that $K$ is said to be \emph{scattered} if every non-empty subset has an isolated point. It is clear that $CB(K)\in\textbf{Ord}$ if and only if $K$ is scattered.  If $K$ is scattered, then $CB(K)=\min\{\xi:K^\xi=\varnothing\}$. We isolate the following easy piece. 

\begin{fact} If $K$ is a compact, Hausdorff space, then $CB(K)$ cannot be a limit ordinal.

\end{fact}

\begin{proof} If  $K$ is a compact, Hausdorff topological space such that $CB(K)$ is a limit ordinal, then $(K\setminus K^\zeta)_{\zeta<CB(K)}$ would be an open cover of $K$ with no finite subcover, which is impossible. 
    
\end{proof}

\begin{rem}\upshape It is worth noting that there are a number of topologies of interest with which one can equip a tree.  We use the same prime notation for derivatives of trees and topological spaces, so there would be potential for confusion if we were topologizing trees with topologies with respect to which the derived trees and Cantor-Bendixson derivatives did not coincide, which in general they need not. In this work, however, we will not endow any trees with a topology.  
\end{rem}

\section{The Grasberg norm and upper estimate}

\begin{definition} Let $K$ be an infinite, scattered, compact, Hausdorff space. Then there exists a unique ordinal $\textbf{o}(K)$ such that $CB(K)\in (\omega^{\textbf{o}(K)}, \omega^{\textbf{o}(K)+1})$. Moreover, there exists a unique $\textbf{b}(K)<\omega$ such that $\omega^{\textbf{o}(K)}\textbf{b}(K)<CB(K)<\omega^{\textbf{o}(K)}(\textbf{b}(K)+1)$.  We define the \emph{Grasberg norm} $|\cdot|$ on $C(K)$ by \[|f|= \max\{\|2^n f|_{K^{\omega^\xi n}}\|:0\leqslant n\leqslant \textbf{b}(K)\}.\]  Obviously $\|\cdot\|\leqslant |\cdot|\leqslant 2^{\textbf{b}(K)}\|\cdot\|$.  

\end{definition}

For convenience, we let $C(\varnothing)$ be the zero vector space and let $f|_\varnothing$ be the zero function.

For an infinite, scattered, compact, Hausdorff space $K$, $f\in C(K)$, $\ee>0$, and $0\leqslant n\leqslant \textbf{b}(K)$,  we let \[\Phi_n(f,\ee)=\Bigl\{\varpi \in K^{\omega^{\textbf{\emph{o}}(K)}n}: 2^{n+1}|f(\varpi)|>|f|+\ee\Bigr\}.\] We let \[\Phi(f,\ee)=\bigcup_{n=0}^{\textbf{b}(K)} \Phi_n(f,\ee).\]

The next two lemmas illustrate the usefulness of the Grasberg norm. The first indicates how the Grasberg norm can be used inductively. The second indicates how we can use the Grasberg norm to find $1$-weakly summing sequences.

\begin{lemma} Let $K$ be an infinite, scattered, compact, Hausdorff space. For any $f\in C(K)$ and $\ee>0$, $CB(\Phi(f,\ee))\leqslant \omega^{\textbf{\emph{o}}(K)}$. 

\label{king}
\end{lemma}

\begin{proof} It is a standard property of the Cantor-Bendixson derivative that for subsets $F_0, \ldots, F_m$ of $K$ and any $\xi$, $\bigl(\bigcup_{n=0}^m F_n\bigr)^\xi=\bigcup_{n=0}^m F_n^\xi$, so $CB\bigl(\bigcup_{n=0}^m F_n\bigr)=\max_{0\leqslant n\leqslant m}CB(F_n)$.  Therefore it is sufficient to show that $CB(\Phi_n(f,\ee))\leqslant \omega^{\textbf{o}(K)}$ for each $0\leqslant n\leqslant \textbf{b}(K)$. To that end, fix $0\leqslant n\leqslant \textbf{b}(K)$.   Seeking a contradiction, assume that \[\varpi\in \Phi_n(f,\ee)^{\omega^{\textbf{o}(K)}}\subset \Phi_n(f,\ee)\cap (K^{\omega^{\textbf{o}(K)}n})^{\omega^{\textbf{o}(K)}}=\Phi_n(f,\ee) \cap K^{\omega^{\textbf{o}(K)}(n+1)}.\]

Then \[|f|\geqslant 2^{n+1}\|f|_{K^{\omega^{\textbf{o}(K)(n+1)}}}\|\geqslant 2^{n+1}|f(\varpi)|>|f|+\ee,\] a contradiction. 
    
\end{proof}

\begin{lemma} Let $K$ be an infinite, scattered, compact, Hausdorff space. Assume that $f,g\in C(K)$ and $\ee>0$ are such that  $\|g|_{\Phi(f,\ee)}\|\leqslant \ee/2^{\textbf{\emph{b}}(K)}$.  Then \[|f+g|\leqslant \max\bigl\{|f|+\ee, |f|/2+\ee/2+|g|\bigr\}.\]   
\label{queen}    
\end{lemma}

\begin{proof} Fix $0\leqslant n\leqslant \textbf{b}(K)$ and $\varpi\in K^{\omega^{\textbf{o}(K)}n}$.   If $\varpi\in \Phi_n(f,\ee)$, then \[2^n|(f+g)(\varpi)|\leqslant 2^n |f(\varpi)|+2^n |g(\varpi)|\leqslant |f|+2^n\ee/2^{\textbf{\emph{b}}(K)}\leqslant |f|+\ee.\] If $\varpi \in K^{\omega^{\textbf{o}(K)}n}\setminus \Phi_n(f,\ee)$, then \[2^n|(f+g)(\varpi)| \leqslant 2^{n+1}|f(\varpi)|/2 + |g| \leqslant |f|/2+\ee/2+|g|.\]

\end{proof}

\subsection{The upper estimate}

\begin{theorem} For any ordinal $\xi$, if $K$ is a compact, Hausdorff space $K$ with $CB(K)\leqslant \omega^\xi$, then $Sz(C(K))\leqslant \omega^\xi$. 

\end{theorem}

\begin{proof} We work by induction on $\xi$.   If $\xi=0$, then $CB(K)\leqslant \omega^0=1$ implies that $K$ is finite, and $C(K)$ is finite-dimensional. Then $Sz(C(K))=1$.  Indeed, for any Banach space $X$ and any norm compact $K\subset B_{X^*}$, the norm and weak$^*$-topologies must coincide, so $s_\ee(K)=\varnothing$ for any $\ee>0$. 

Next, assume that for some limit ordinal $\xi$, $CB(K)\leqslant \omega^\xi$.   Since $\omega^\xi$ is a limit ordinal and $CB(K)$ is not, and since $\sup_{\zeta<\xi} \omega^\zeta=\omega^\xi$ because $\xi$ is a limit ordinal, it follows that for some $\zeta<\xi$, $CB(K)\leqslant \omega^\zeta$. Then  $Sz(C(K))\leqslant \omega^\zeta\leqslant \omega^\xi$ by the inductive hypothesis.  

Last, assume that the result holds for $\xi$ and $CB(K)\leqslant \omega^{\xi+1}$.   Again, since $\omega^{\xi+1}$ is a limit ordinal and $CB(K)$ is not, we must have $CB(K)<\omega^{\xi+1}$. If $CB(K)\leqslant \omega^\xi$, then we conclude by the inductive hypothesis that $Sz(C(K))\leqslant \omega^\xi<\omega^{\xi+1}$, so it remains to consider the case that $\omega^\xi<CB(K)<\omega^{\xi+1}$. That is, it remains to consider the case that $\textbf{o}(K)=\xi$.  Fix $\delta>0$ and $n\in\nn$ such that $n \delta>2^{2+\textbf{b}(K)}$.   Fix $\ee>0$ such that $(1+\ee)^n<2$. Fix a tree $T$ with $\text{rank}(T)=\omega^{\xi+1}$ and a weakly null collection $(f_t)_{t\in T}\subset B_{C(K)}$.  Let $S=T\setminus T^{\omega^\xi n}$ and note that $\text{rank}(S)=\omega^\xi n$ and $(f_t)_{t\in S}$ is weakly null.   Let $s_0=\varnothing$ and fix $s_1\in MAX(S^{\omega^\xi(n-1)})$ arbitrary.  Next, assume that for some $m<n$, $\varnothing=s_0\prec \ldots \prec s_m$ have been chosen.  Suppose also that for each $1\leqslant i\leqslant m$, $g_i\in \text{co}\{f_s:s_{i-1}\prec s\preceq s_i\}$ and $|g|\leqslant (1+\ee)^{m-1}$, where $g=\frac{1}{2^{1+\textbf{b}(K)}}\sum_{i=1}^m g_i$.   Let $R=\{t\in S^{\omega^\xi(n-m-1)}:s_m\prec t\}$ and note that $\text{rank}(R)=\omega^\xi$ and $(f_t)_{t\in R}$ is weakly null.    By Lemma \ref{king}, $CB(\Phi(g,\ee))\leqslant \omega^\xi$.  By the inductive hypothesis, $Sz(C(\Phi(g,\ee)))\leqslant \omega^\xi$, from which it follows that there exist $s_{m+1}\in R$ and \[g_{m+1}\in \text{co}\{f_s:s\in R, s\preceq s_{m+1}\}= \text{co}\{f_s:s\in S, s_m\prec s\preceq s_{m+1}\}\] such that $\|g|_{\Phi(g,\ee)}\|<\ee/2^{\textbf{b}(K)}$.    By Lemma \ref{queen},  \begin{align*} \Bigl|\frac{1}{2^{1+\textbf{b}(K)}}\sum_{i=1}^{m+1} g_i\Bigr| & \leqslant \max\{|g|+\ee, |g|/2+\ee/2+|g_{m+1}/2| \\ & \leqslant \max\Bigl\{(1+\ee)^{m-1}+\ee, \frac{(1+\ee)^{m-1}+\ee}{2}+\frac{1}{2}\Bigr\} < (1+\ee)^m.  \end{align*}  This completes the recursive construction.   

Define $f=\frac{1}{n}\sum_{i=1}^n f_i\in \text{co}\{f_s:s\preceq s_n\}$ and note that \[\|f\|\leqslant |f|= \frac{2^{1+\textbf{b}(K)}}{n}\Bigl|\frac{1}{2^{1+\textbf{b}(K)}}\sum_{i=1}^n g_i\Bigr|\leqslant \frac{2^{1+\textbf{b}(K)}(1+\ee)^n}{n} < \frac{2^{2+\textbf{b}(K)}}{n}<\delta.\] Since $\delta>0$ was arbitrary, we are done.

\end{proof}

\section{The lower estimate and main result}

We obtain the easy lower estimates. 

\begin{proposition} Let $K$ be a compact, Hausdorff topological space.  Then $Sz(C(K))\geqslant CB(K)$. 
\label{lower}
\end{proposition}

\begin{proof} Define $\varphi:K\to D\subset B_{C(K)^*}$ by letting $\varphi(\varpi)=\delta_\varpi$, where $\delta_\varpi$ is the Dirac evaluation functional given by $\delta_\varpi(f)=f(\varpi)$ and $D=\varphi(K)$.  Note that $\varphi$ is a homeomorphism of $K$ with $(D, \text{weak}^*)$.   We claim that for any $\ee\in (0,2)$ and any ordinal $\xi$, \[\varphi(K^\xi) \subset s_\ee^\xi(D).\] The $\xi=0$ case follows by definition and the limit ordinal case follows by taking intersections.  Assume $\varphi(K^\xi)\subset s_\ee^\xi(K)$ and $\varpi\in K^{\xi+1}$.  This means there exists a net $(\varpi_\lambda)\subset K^\xi\setminus \{\varpi\}$ which converges to $\varpi$. Then $(\delta_{\varpi_\lambda})$ is a net in $s^\xi_\ee(K)$ which is weak$^*$-convergent to $\delta_\varpi$.  Since $\|\delta_\varpi-\delta_{\varpi_\lambda}\|=2>\ee$ for all $\lambda$, it follows that $\text{diam}(V\cap K^\xi)=2>\ee$ for any weak$^*$-neighborhood $V$ of $\delta_\varpi$. This yields that $\varphi(\varpi)=\delta_\varpi\in K^{\xi+1}$.

\end{proof}

\begin{proof}[Proof of Theorem \ref{main}]  We know from Proposition \ref{lower} that $Sz(C(K))\geqslant CB(K)$, which means that if $K$ is not scattered, then $Sz(C(K))\geqslant CB(K)=\textbf{Ord}$. On the other hand, if $K$ is scattered, then $CB(K)<\omega^\xi$ for some $\xi$, and $Sz(C(K))\leqslant \omega^\xi<\textbf{Ord}$. Therefore $C(K)$ is Asplund if and only if $K$ is scattered.  

Assume $K$ is scattered.   We know from Proposition \ref{form} that $Sz(C(K))\in \{\omega^\xi:\xi\in \textbf{Ord}\}$, and we know from Proposition \ref{lower} that  $Sz(C(K))\geqslant CB(K)$.  So $Sz(C(K))\geqslant \min\{\omega^\xi:\omega^\xi\geqslant CB(K)\}$.   

But we know from Theorem \ref{upper} that $Sz(C(K))\leqslant \min\{\omega^\xi:\omega^\xi\geqslant CB(K)\}$, so $Sz(C(K))=\min\{\omega^\xi:\omega^\xi\geqslant CB(K)\}$.
    
\end{proof}

\end{document}